\documentclass[12pt]{article}

\usepackage{latexsym}
\usepackage{graphics}
\usepackage{amsmath}
\usepackage{xspace}
\usepackage{amssymb}
\usepackage{psfrag}
\usepackage{epsfig}
\usepackage{amsthm}
\usepackage{pst-all}

\usepackage{color}

\definecolor{Red}{rgb}{1,0,0}
\definecolor{Blue}{rgb}{0,0,1}
\definecolor{Olive}{rgb}{0.41,0.55,0.13}
\definecolor{Yarok}{rgb}{0,0.5,0}
\definecolor{Green}{rgb}{0,1,0}
\definecolor{MGreen}{rgb}{0,0.8,0}
\definecolor{DGreen}{rgb}{0,0.55,0}
\definecolor{Yellow}{rgb}{1,1,0}
\definecolor{Cyan}{rgb}{0,1,1}
\definecolor{Magenta}{rgb}{1,0,1}
\definecolor{Orange}{rgb}{1,.5,0}
\definecolor{Violet}{rgb}{.5,0,.5}
\definecolor{Purple}{rgb}{.75,0,.25}
\definecolor{Brown}{rgb}{.75,.5,.25}
\definecolor{Grey}{rgb}{.5,.5,.5}

\newcommand{\bfX}{{\bf X}}
\newcommand{\bfx}{{\bf x}}
\newcommand{\bfY}{{\bf Y}}
\newcommand{\bfZ}{{\bf Z}}

\newcommand{\pr}{\mathbb{P}}
\newcommand{\E}{\mathbb{E}}
\newcommand{\R}{\mathbb{R}}

\newcommand{\G}{\mathbb{G}}

\newcommand{\T}{\mathbb{T}}

\newcommand{\I}{{\bf \mathcal{I}}}

\newcommand{\Overlap}{\mathcal{\mathcal{O}}}
\newcommand{\ROverlap}{\mathcal{\mathcal{R}}}

\newcommand{\mb}[1]{\mbox{\boldmath $#1$}}

\newcommand{\ignore}[1]{\relax}

\newtheorem{theorem}{Theorem}[section]

\newtheorem{lemma}[theorem]{Lemma}
\newtheorem{conj}[theorem]{Conjecture}

\newtheorem{proposition}[theorem]{Proposition}

\newcommand{\ER}{Erd{\"o}s-R\'{e}nyi }
\newcommand{\Lovasz}{Lov\'{a}sz}

\definecolor{Red}{rgb}{1,0,0}
\definecolor{Blue}{rgb}{0,0,1}
\definecolor{Olive}{rgb}{0.41,0.55,0.13}
\definecolor{Green}{rgb}{0,1,0}
\definecolor{MGreen}{rgb}{0,0.8,0}
\definecolor{DGreen}{rgb}{0,0.55,0}
\definecolor{Yellow}{rgb}{1,1,0}
\definecolor{Cyan}{rgb}{0,1,1}
\definecolor{Magenta}{rgb}{1,0,1}
\definecolor{Orange}{rgb}{1,.5,0}
\definecolor{Violet}{rgb}{.5,0,.5}
\definecolor{Purple}{rgb}{.75,0,.25}
\definecolor{Brown}{rgb}{.75,.5,.25}
\definecolor{Grey}{rgb}{.5,.5,.5}
\definecolor{Pink}{rgb}{1,0,1}
\definecolor{DBrown}{rgb}{.5,.34,.16}
\definecolor{Black}{rgb}{0,0,0}

\author{
{\sf David Gamarnik}\thanks{MIT; e-mail: {\tt gamarnik@mit.edu}.Research supported  by the NSF grants CMMI-1031332.}
\and
{\sf Madhu Sudan}\thanks{Microsoft Research New England; e-mail: {\tt madhu@mit.edu}}
}

\begin{document}

\title{Limits of local algorithms over sparse random graphs}
\date{}

\maketitle

\begin{abstract}
Local algorithms on graphs are algorithms that run in parallel on
the nodes of a graph to compute some global structural feature
of the graph. Such algorithms use only local information available
at nodes to determine local aspects of the global structure, while
also potentially using some randomness.
Recent research has shown that such
algorithms show significant promise in computing structures like large
independent sets in graphs locally.
Indeed the promise led to a conjecture
by Hatami, \Lovasz~and Szegedy~\cite{HatamiLovaszSzegedy} that local algorithms
may be able to compute maximum independent sets in (sparse) random
$d$-regular graphs.
In this paper we refute this conjecture and show that every independent set produced by local algorithms
is multiplicative factor $1/2+1/(2\sqrt{2})$ smaller than the largest, asymptotically as $d\rightarrow\infty$.

Our result is based on an important clustering phenomena predicted first in the literature on spin glasses,
and recently proved rigorously for a variety of constraint satisfaction problems on random graphs.
Such properties suggest that the geometry of the solution space can be
quite intricate.
The specific clustering property,
that we prove and apply in this paper
shows that typically every two large independent sets in a random graph
either have a significant intersection, or have a nearly empty
intersection. As a result, large independent sets are clustered according to the proximity to each other. While
the clustering property was postulated earlier as an obstruction for the success of local algorithms, such as for example, the Belief Propagation algorithm,
our result is the first one where the clustering property is used to
formally prove limits on local algorithms.
\end{abstract}

\section{Introduction}
Local algorithms are decentralized algorithms that run in parallel
on nodes in a network using only information available from
local neighborhoods to compute some global function of data that
is spread over the network. Local algorithms have been studied in
the past in various communities. They arise as natural solution concepts
in distributed computing (see, e.g., ~\cite{Linial}).
They also lead to efficient sub-linear
algorithms --- algorithms that run in time significantly less than the
length of the input --- and \cite{ParnasRon,NguyenOnak,hassidimetal,ronittetal}
illustrate some of the progress in this direction.
Finally local algorithms have also been
proposed as natural heuristics
for solving hard optimization problems with the popular Belief Propagation
algorithm (see for instance ~\cite{JordanWainwright,MezardMontanariBook})
being one such example.

In this work we study the performance of a natural class of local
algorithms on {\em random regular graphs} and show limits on the
performance of these algorithms.
The motivation for our work comes from the
a notion of local algorithms that has appeared in a completely
different mathematical context, namely that of the theory of graph
limits, developed in several papers,
including~\cite{BorgsChayesEtAlGraphLimitsI},\cite{BorgsChayesEtAlGraphLimitsII},\cite{LovaszSzegedy},\cite{BorgsChayesKahnLovasz},\cite{borgs2013convergent},\cite{elek2010borel},
\cite{HatamiLovaszSzegedy}.
In the realms of this theory it was conjectured that every
``reasonable" combinatorial optimization problem on {\em random graphs}
can be solved by means of some local algorithms.
To the best of our knowledge this conjecture for the first time was formally stated in Hatami, \Lovasz~ and Szegedy in~\cite{HatamiLovaszSzegedy},
and thus, from now on, we will refer to it as Hatami-\Lovasz-Szegedy (or HLS) conjecture, though informally it was posed by Szegedy earlier, and was referenced in several papers,
including Lyons and Nazarov~\cite{lyons2011perfect}, and Cs\'{o}ka and Lippner~\cite{csoka2012invariant}.
In a concrete context of the problem of finding largest independent sets
in sparse random regular graphs, the conjecture is stated as follows. Let $\T_{d,r}$ be a rooted $d$-regular tree with depth $r$. Namely, every node including the root,
has degree $r$, except for the leaves, and the distance from the root to every leaf is $r$. Consider a function $f_r:[0,1]^{\T_{d,r}}\rightarrow \{0,1\}$
which maps every such tree whose nodes are decorated
with real values from $[0,1]$ to a "decision" encoded by $0$ and $1$. In light of the fact that in a random $d$-regular graph $\G_d(r)$
the typical node has depth-$r$ neighborhood isomorphic to $\T_{d,r}$, for any constant $r$,
such a function $f_r$ can be used to generate (random) subsets $I$ of $\G_d(r)$ as follows: decorate nodes of $\G_d(r)$ using i.i.d. uniform random values from $[0,1]$ and
apply function $f_r$ in every node. The set of nodes for which $f_r$ produces value $1$ defines $I$, and is called "i.i.d. factor".
It is clear that $f_r$ essentially describes a local algorithm for producing sets $I$ (sweeping issue of computability of $f_r$ under the rug).
The HLS conjecture postulates the existence of
a sequence of $f_r, r=1,2,\ldots$, such that the set $I$ thus produced is an independent subset of $\G_d(r)$ and asymptotically achieves the largest possible value as $r\rightarrow\infty$.
Namely, largest independent subsets of random regular graphs are i.i.d. factors.
The precise connection between this conjecture and the theory of graph limits  is beyond the scope of this
paper. Instead we refer the reader to the relevant papers~\cite{HatamiLovaszSzegedy},\cite{elek2010borel}. The concept of i.i.d. factors appears also in one of the open
problem by David Aldous~\cite{Aldous:FavoriteProblemsNew} in the context of coding invariant processes on infinite trees.

It turns out that an analogue for the HLS conjecture is indeed valid for another important combinatorial
optimization problem - matching. Lyons and Nazarov~\cite{lyons2011perfect}
established it for the case of bi-partite locally $\T_{d,r}$-tree-like graphs, and Cs\'{o}ka and Lippner established this result for general locally $\T_{d,r}$-tree-like graphs.
Further, one can modify the framework of i.i.d. factors by encapsulating non-$\T_{d,r}$ type neighborhoods, for example by making $f_r$ depend not only on the realization
of random uniform in $[0,1]$ values, but also on the realization of the graph-theoretic neighborhoods around the nodes.
Some probabilistic bound on a degree might be needed to make this definition rigorous
(though we will not attempt this formalization in this paper). In this case one can consider, for example, i.i.d. factors when  neighborhoods
are distributed as $r$ generations of a branching
process with Poisson distribution, and then ask which combinatorial optimization problems defined now on sparse \ER graphs $\G(n,d/n)$ can be solved as i.i.d. factors.
Here $\G(n,d/n)$ is a random graph on $n$ nodes with each of the ${n\choose 2}$ edges selected with probability $d/n$, independently for all edges, and $d>0$ is a fixed constant.
In this case it is possible to show that when $c\le e$, the maximum independent set problem on $\G(n,d/n)$ can be solved nearly optimally by the well known Belief Propagation (BP)
algorithm with constantly many rounds. Since the BP is a local algorithm, then the maximum independent set on $\G(n,d/n)$ is an i.i.d.
factor, in the extended framework defined above.
(We should note that the original proof of Karp and Sipser~\cite{KarpSipser} of the very similar result, relied on a different method.)
Thus, the framework of local algorithms viewed as i.i.d. factors is rich enough to solve several interesting combinatorial optimization problems.

Nevertheless, in this paper we refute the HLS conjecture in the context of maximum independent set problem on random regular graphs $\G_d(n)$.
Specifically, we show that for large enough $d$, with high probability as $n\rightarrow\infty$,
every independent set producible as an i.i.d. factor  is a multiplicative factor $\gamma<1$ smaller than
the largest independent subset of $\G_d(n)$. We establish that $\gamma$ is asymptotically at most ${1\over 2}+{1\over 2\sqrt{2}}$ (though we conjecture that the result
holds simply for $\gamma=1/2$, as we discuss in the body of the paper).

Our result is based on a powerful, though fairly simple to establish in our case,  so-called
\emph{clustering} or \emph{shattering} property of some combinatorial optimization problem on random graphs, first conjectured
in the theory of spin glasses and later confirmed by rigorous means. For the first time
this clustering property was discussed in terms of so-called overlap structure of the solutions of the Sherrington-Kirkpatrick model~\cite{TalagrandBook}. Later, it featured
in the context of random K-SAT problem and was proved rigorously by Achlioptas, Coja-Oghlan and Ricci-Tersenghi~\cite{AchlioptasCojaOghlanRicciTersenghi},
and by Mezard, Mora and Zecchina~\cite{mezard2005clustering}, independently. We do not define the random K-SAT problem here and instead refer the reader to the aforementioned papers.
What these results state is that in certain regimes, the set of satisfying assignments, with high probability, can be clustered into groups such that two solutions within the same
cluster agree on a certain minimum number of variables,
while two solutions from different clusters have to disagree on a certain minimum number of variables. In particular, one can identify a certain
non-empty interval $[z_1,z_2]\subset [0,1]$ such that \emph{no} two solutions of the random K-SAT problem agree on precisely $z$ fraction of variables for all $z\in [z_1,z_2]$.
One can further show that the onset of clustering property occurs when the density of clauses to variables becomes at least $2^K/K$, while at the same time
the formula remains satisfiable
with high probability, when the density is below $2^K\log 2$. Interestingly, the known algorithms for finding solutions of random instances of K-SAT problem
also stop working around the $2^K/K$ threshold. It was widely conjectured that the onset of the clustering phase is the main obstruction for finding such algorithms.
In fact, Coja Oghlan~\cite{coja2011belief} showed that the BP algorithm, which was earlier conjectured to be a good contender for solving the random instances of K-SAT problems,
also fails when the density of clauses to variables is at least $2^K\log K/K$, though Coja-Oghlan's approach does not explicitly rely on the clustering
property, and one could argue that the connection between the clustering property and the failure of the BP algorithm is coincidental.

Closer to the topic of this paper, the
clustering property was also recently established for independent sets in \ER graphs. In particular Coja-Oghlan and Efthymiou~\cite{coja2011independent}
established the following result. It is known that the largest independent subset of $\G(n,d/n)$ has size approximately
$(2\log d/d)n$, when $d$ is large, see the next section for precise details. The authors of~\cite{coja2011independent} show that
the set of independent sets of size at least approximately $(\log d/d)n$ (namely those within factor $1/2$ of the optimal), are also clustered. Namely, one
can split them into groups such that intersection of two independent sets within a group has a large cardinality, while intersection of two independent sets
from different groups has a small cardinality. One should note that algorithms for producing large independent subsets of random graphs also stop short
factor $1/2$ of the optimal, both in the case of sparse and in the dense random graph cases,
as exhibited by the well-known Karp's open problem regarding independent subsets of $\G(n,1/2)$~\cite{AlonSpencer}.

This is almost the result we need for our analysis with two exceptions. First, we need to establish this clustering property for random regular as opposed \ER graphs.
Second, the result in~\cite{coja2011independent} applies to \emph{typical} independent sets and does not rule out the possibility that there are two independent sets
with some "intermediate" intersection cardinality, though the number of such pairs is insignificant compared to the total number of
independent sets. For our result we need to show that, without exception, every
pair of "large" independent sets has either large or small intersection. We indeed establish this, but at the cost of loosing additional factor $1/(2\sqrt{2})$.
In particular, we show that for large enough $d$, with high probability as $n\rightarrow\infty$,
every two independent subsets of $\G_d(n)$ with cardinality asymptotically $(1+\beta)(\log d/d)n$, where $1\ge \beta> {1\over 2}+{1\over 2\sqrt{2}}$ either
have intersection size at least $(1+z)(\log d/d)n$ or at most $(1-z)(\log d/d)n$, for some $z<\beta$. The result is established using a straightforward
first moment argument: we compute the expected number of pairs of independent sets with intersection lying in the interval $[(1-z)(\log d/d)n,(1+z)(\log d/d)n]$,
and show that this expectation converges to zero exponentially fast.

With this result at hand, the refutation of the HLS conjecture is fairly simple to derive. We prove that if local algorithms can construct
independent sets of size asymptotically $(1+\beta)(\log d/d)n$, then, by means of a simple coupling construction, we can construct two independent sets with intersection
size $z$ for \emph{all} $z$ in the interval $[(1+\beta)^2(\log d/d)^2n,(1+\beta)(\log d/d)n]$, clearly violating the clustering property.
The additional factor $1/(2\sqrt{2})$ is an artifact of the analysis, and hence we believe that our result holds for all $\beta\in (0,1]$. Namely, no local
algorithm is capable of producing independent sets with size larger than factor $1/2$ of the optimal, asymptotically in $d$. We note again that this coincides
with the barrier for known algorithms. It is noteworthy that our result is the first one where algorithmic hardness derivation relies directly on the
the geometry of the solution space, vis-a-vis the clustering phenomena, and thus the connection between algorithmic hardness and clustering property is not coincidental.

The remainder of the paper is structured as follows. We introduce
some basic material  and the HLS conjecture
in the next section. In the same
section we state our main theorem --- non-validity of the conjecture
(Theorem~\ref{theorem:MainResult}).
We also state two secondary theorems, the first describing the
overlap structure of independent sets in random graphs (Theorem~\ref{thm:cluster}) - the main tool in the proof of our result, and
the second describing overlaps that can be found if local algorithms
work well (Theorem~\ref{thm:couple}).
We prove our main theorem easily from the two secondary theorems
in Section~\ref{section:ProofOfMainResult}.
We prove Theorem~\ref{thm:couple} in Section~\ref{sec:couple}.
Sections~\ref{section:Proof.ER} and~\ref{section:RandomRegGraph} are devoted
to proofs of the theorem regarding the overlap property, for the case of \ER and random regular graph, respectively. While technically we do not need such a result for the \ER
graph, it is very simple to derive and provides the roadmap for the case
of the regular graphs (where the calculations are a bit more tedious).
The \ER case might also be useful for further studies of i.i.d. factors
on \ER graphs as opposed to random regular graphs, in the framework
described above.


\section{Preliminaries and main result}\label{section:preliminaries}
For convenience, we repeat here some of the notions and definitions already introduced in the first section.

\paragraph{Basic graph terminology}
All graphs in this paper are understood to be simple undirected graphs.
Given a graph $\G$ with node set $V(\G)$ and edge set
$E(\G)$, a subset of nodes $I\subset V(\G)$ is an independent set if $(u,v)\notin E(\G)$ for all $u,v\in I$.
A path between nodes $u$ and $v$ with length $r$ is a sequence of nodes $u_1,\ldots,u_{r-1}$ such that $(u,u_1),(u_1,u_2),\ldots,(u_{r-1},v)\in E(\G)$.
The distance between nodes $u$ and $v$ is the length of the shortest path between them.
For every positive integer value $r$ and every node $u\in V(\G)$, $B_{\G}(u,r)$ denotes the depth-$r$ neighborhood of $u$ in $\G$. Namely,
$B_{\G}(u,r)$ is the subgraph of $\G$ induced by nodes $v$ with distance at most $r$ from $u$.
When $\G$ is clear from context we drop the subscript.
The degree of a vertex $u \in V(\G)$ is the number of vertices
$v$ such that $(u,v) \in E(\G)$. The degree of a graph $\G$ is the
maximum degree of a vertex of $\G$.
A graph $\G$ is $d$-regular if the degree of every node is $d$.

\paragraph{Random graph preliminaries}
Given a positive real $d$,
$\G(n,d/n)$ denotes the \ER graph on $n$ nodes $\{1,2,\ldots,n\}\triangleq [n]$,
with edge probability $d/n$. Namely each of the ${n\choose 2}$ edges of a complete graph on $n$ nodes belongs
to $E(\G(n,d/n))$ with probability $d/n$, independently for all edges. Given a positive integer $d$, $\G_d(n)$ denotes
a graph chosen uniformly at random from the space of all $d$-regular graphs on $n$ nodes. This definition is meaningful only
when $nd$ is an even number, which we assume from now on. Given a positive integer $m$, let $\I(n,d,m)$ denote the
set of all independent sets in $\G(n,d/n)$ with cardinality $m$. $\I_d(n,m)$ stands for a similar set for the case of random regular graphs.
Given integers $0\le k\le m$, let $\Overlap(n,d,m,k)$ denote the set of pairs $I,J\in\I(n,d,m)$ such that $|I\cap J|=k$. The definition of the
set $\Overlap_d(n,m,k)$ is similar. The  sizes of the sets $\Overlap(n,d,m,k)$ and $\Overlap_d(n,m,k)$, and in particular whether these sets are empty
or not, is one of our focuses.

Denote by $\alpha(n,d)$ the size of a largest in cardinality independent subset of $\G(n,d/n)$, normalized by $n$. Namely,
\begin{align*}
\alpha(n,d)=n^{-1}\max\{m: \I(n,d,m)\ne\emptyset\}.
\end{align*}
$\alpha_d(n)$ stands for the similar quantity for random regular graphs. It is known that $\alpha(n,d)$ and $\alpha_d(n)$ have deterministic limits as $n\rightarrow\infty$.
\begin{theorem}\label{theorem:LargestIS}
For every $d\in\R_+$ there exists $\alpha(d)$ such that w.h.p. as $n\rightarrow\infty$,
\begin{align}\label{eq:BGT}
\alpha(n,d)\rightarrow \alpha(d).
\end{align}
Similarly, for every positive integer $d$ there exists $\alpha_d$ such that w.h.p. as $n\rightarrow\infty$
\begin{align}\label{eq:BGTregular}
\alpha_d(n)\rightarrow \alpha_d.
\end{align}
Furthermore
\begin{align}
\alpha(d)&={2\log d\over d}(1-o(1)) \label{eq:FriezeLimit},\\
\alpha_d&={2\log d\over d}(1-o(1))\label{eq:FriezeLimitReg},
\end{align}
as $d\rightarrow\infty$.
\end{theorem}
The convergence (\ref{eq:BGT}) and (\ref{eq:BGTregular}) was established in Bayati, Gamarnik and Tetali~\cite{BayatiGamarnikTetali}. The limits
(\ref{eq:FriezeLimit}) and (\ref{eq:FriezeLimitReg}) follow from much older results by Frieze~\cite{FriezeIndependentSet} for the case of \ER graphs and by
Frieze and {\L}uczak~\cite{frieze1992independence} for the case of random regular graphs, which established these limits in the $\limsup_n$ and $\liminf_n$
sense.
The fallout of these results is that graphs $\G(n,d/n)$ and $\G_d(n)$
have independent sets of size up to approximately $(2\log d/d)n$, when $n$ and $d$ are large, namely in the doubly asymptotic sense when we first take $n$ to infinity
and then $d$ to infinity.

\paragraph{Local graph terminology}
A {\em decision function}
is a measurable function
$f = f(u,\G,\bfx)$ where
$\G$ is a graph on vertex set $[n]$ for some positive integer $n$,
$u \in [n]$ is a vertex and
$\bfx \in [0,1]^N$ is a sequence of real numbers for some $N \geq n$ and returns a
Boolean value $\{0,1\}$.
A decision function $f$ is said to compute an independent
set if for every graph $\G$ and every sequence $\bfx$ and for
every pair $(u,v) \in E(\G)$ it is the case that either
$f(u,\G,\bfx)= 0$ or $f(v,\G,\bfx) = 0$, or both.
We refer to such an $f$ as an independence function.
For an independence function $f$, graph $\G$ on vertex
set $[n]$ and $\bfx \in [0,1]^N$ for $N \geq n$,
we let $I_\G(f,\bfx)$ denote the independent set of $\G$ returned
by $f$, i.e., $I_\G(f,\bfx) = \{u \in [n] \mid f(u,\G,\bfx) = 1\}$. We will assume later that $X$ is chosen randomly according to some probability distribution.
In this case $I_\G(f,\bfx)$ is a randomly chosen independent set in $\G$.

We now define the notion of a ``local'' decision function, i.e., one whose actions depend only on the local structure of a graph and the local randomness. The definition is a natural one, but we formalize it below for completeness.
Let $\G_1$ and $\G_2$ be graphs on vertex sets $[n_1]$ and $[n_2]$ respectively. Let $u_1 \in [n_1]$ and $u_2 \in [n_2]$.
We say that $\pi:[n_1] \to [n_2]$ is an $r$-local isomorphism mapping $u_1$ to $u_2$ if $\pi$ is a graph isomorphism from $B_{\G_1}(u_1,r)$ to $B_{\G_2}(u_2,r)$ (so in particular it is a bijection from $B_{\G_1}(u_1,r)$ to $B_{\G_2}(u_2,r)$, and further it preserves adjacency within $B_{\G_1}(u_1,r)$ and $B_{\G_2}(u_2,r)$).
For $\G_1,\G_2,u_1,u_2$ and an $r$-local isomorphism $\pi$, we say sequences $x^{(1)} \in [0,1]^{N_1}$ and $x^{(2)} \in [0,1]^{N_2}$ are $r$-locally equivalent if for every $v \in B_{\G_1}(u_1,r)$ we have $x^{(1)}_v = x^{(2)}_{\pi(v)}$.
Finally we say $f(u,\G,x)$ is an $r$-local function if for every pair of graphs $\G_1,\G_2$, for every pair of vertices
$u_1 \in V(\G_1)$ and $u_2 \in V(\G_2)$, for every $r$-local isomorphism $\pi$ mapping $u_1$ to $u_2$ and $r$-locally equivalent sequences $x^{(1)}$ and $x^{(2)}$ we have
$f(u_1,\G_1,x^{(1)}) =
f(u_2,\G_2,x^{(2)})$.
We often use the notation $f_r$ to denote an $r$-local function.

Let $n_{d,r} \triangleq 1 + d \cdot ((d-1)^r - 1)/(d-2)$ denote
the number of vertices in a rooted tree of degree $d$ and depth
$r$.
We let $\T_{d,r}$ denote a canonical rooted tree on vertex set $[n_{d,r}]$ with
root being $1$. For $n \geq n_{d,r}, \bfx \in [0,1]^n$ and
an $r$-local function $f_r$, we let $f_r(\bfx)$ denote the
quantity $f_r(1,\T_{d,r},\bfx)$.
Let $\bfX$ be chosen according to a uniform distribution on $[0,1]^n$. The set subset of nodes $I_{\G_d(n)}(f_r,\bfX)$
is called \emph{i.i.d. factor} produced by the $r$-local function $f_r$.
As we will see below the $\alpha(f_r) \triangleq \frac1n \cdot \E_{\bfX}[f_r(\bfX)]$ accurately
captures (to within an additive $o(1)$ factor) the density of an independent
returned by an $r$-local independence function $f_r$ on $\G_d(n)$.

First we recall the following folklore proposition which we will also use often in this paper.
\begin{proposition}
\label{prop:local-tree}
As $n \to \infty$, with probability tending to $1$ almost all local neighborhoods
in $\G_d(n)$ look like a tree. Formally, for every $d$, $r$ and $\epsilon$, for sufficiently
large $n$,
$$\pr_{\G_d(n)} \left( | \{u \in [n] \mid B_{\G_d(n)}(u,r) \not\cong \T_{d,r}  \}| \geq \epsilon n \right) \leq \epsilon.$$
\end{proposition}
This immediately implies that the expected value of the independent set $I_{\G_d(n)}(f_r,\bfX)$ produced by $f_r$ is $\alpha(f_r)n+o(n)$.
In fact the following concentration result holds.

\begin{proposition}
\label{prop:variance}
As $n \to \infty$, with probability tending to $1$ the independent set
produced by a $r$-local function $f$ on $\G_d(n)$ is of size $\alpha(f) \cdot n + o(n)$.
Formally, for every $d$, $r$, $\epsilon$ and every $r$-local function $f$,
for sufficiently large $n$,
$$\pr_{\G_d(n),\bfX \in [0,1]^N} \left( | |I_{\G_d(n)}(f_r,\bfX)| - \alpha(f_r)n | \geq \epsilon n \right) \leq \epsilon.$$
\end{proposition}

\begin{proof}
The proof follows from by the fact that the variance of $|I_{\G_d(n),\bfX}|$ is $O(n)$
and its expectation is $\alpha(f_r)n+o(n)$, and so the concentration follows
by Chebychev's inequality. The bound on the variance in turn follows from the fact that
for every graph $\G$, there are at most $O(n)$ pairs of vertices $u$ and $v$ for which
the events
$f(u,\G,\bfX)$ and $f(v,\G,\bfX)$ are not independent for random $\bfX$. Details omitted.
\end{proof}

\paragraph{The Hatami-\Lovasz-Szegedy Conjecture and our result}

We now turn to describing the Hatami-\Lovasz-Szegedy (HLS) conjecture and our result.
Recall $\alpha_d$ defined by (\ref{eq:BGTregular}). The HLS
conjecture can be stated as follows.

\begin{conj}\label{conjecture:HatamiLovaszSzegedy}
There exists a sequence of $r$-local independence functions
$f_r, r\ge 1$ such that almost surely
$I(f_r,n)$ is an independent set in $\G_d(n)$ and  $\alpha(f_r)\rightarrow \alpha_d$ as $r\rightarrow\infty$.
\end{conj}

Namely, the conjecture asserts the existence of a local algorithm ($r$-local independence
function $f_r$) which is capable of producing independent sets in $\G_d(r)$
of cardinality close to the largest that exist.
For such an algorithm to be efficient the function $f_r(u,\G,\bfx)$ should also
be efficiently computable {\em uniformly}.
Even setting this issue aside, we show
that there is a limit on the power of local algorithms to find large independent sets
in $\G_d(n)$ and in particular the HLS conjecture does not hold.
Let $\hat\alpha_d=
\sup_r\sup_{f_r}\alpha(f_r)$, where the second supremum is taken
over all $r$-local independence functions $f_r$.

\begin{theorem}\label{theorem:MainResult}[Main]
For every $\epsilon>0$ and all sufficiently large $d$,
\begin{align*}
{\hat\alpha_d\over \alpha_d}\le {1\over 2}+{1\over 2\sqrt{2}}+\epsilon.
\end{align*}
That is, for every $\epsilon>0$ and for all sufficiently large $d$, a largest independent set obtainable by $r$-local functions is at most
${1\over 2}+{1\over 2\sqrt{2}}+\epsilon$ for all $r$.
\end{theorem}

Thus for all large enough $d$ there is a multiplicative gap between
$\hat\alpha_d$ and the independence ratio $\alpha_d$. That being said, our result does not rule
out that for small $d$, $\hat\alpha_d$ in fact equals $\alpha_d$, thus leaving the HLS conjecture open in this regime.

The two main ingredients in our proof of Theorem~\ref{theorem:MainResult}
both deal with the {\em overlaps} between independent sets in random
regular graphs. Informally, our first result on the size of
the overlaps shows that in random graphs the overlaps are not
of ``intermediate'' size --- this is formalized in
Theorem~\ref{theorem:ISclustering}.
We then show that we can apply any $r$-local function $f_r$ twice,
with coupled randomness, to produce two independent sets
of intermediate overlap where the size of the overlap depends on
the size of the independent sets found by $f_r$ and the level of
coupling. This is formalized in Theorem~\ref
{thm:couple}
Theorem~\ref{theorem:MainResult} follows immediately by combinig
the two theorems (and appropriate setting of parameters).

\paragraph{Overlaps in random graphs}

We now state our main theorem about the overlap of large independent
sets. We interpret the statement after we make the formal statement.

\begin{theorem}\label{theorem:ISclustering}
\label{thm:cluster}
For $\beta\in (1/\sqrt{2},1)$ and
$0 < z < \sqrt{2\beta^2 - 1}<\beta$ and
$d$, let $s = (1 + \beta)d^{-1}\log d$ and
let $K(z)$ denote the set of integers between
${(1 - z)n\log d\over d}$ and ${(1+z)n\log d \over d}$.
Then, for all large enough $d$, we have
\begin{align}\label{eq:NonOverlap}
\lim_{n\rightarrow\infty}\pr\Big( \cup_{k\in K(z)}  \Overlap(n,d,\lfloor sn \rfloor,k)\ne\emptyset \Big)=0,
\end{align}
and
\begin{align}\label{eq:NonOverlapRegular}
\lim_{n\rightarrow\infty}\pr\Big(\cup_{k\in K(z)}
\Overlap_d(n,\lfloor sn \rfloor,k)\ne\emptyset \Big)=0.
\end{align}
\end{theorem}
In other words, both in the \ER and in the random regular graph models, when $\beta>1/\sqrt{2}$, and $d$ is large enough,
with probability approaching unity as $n\rightarrow\infty$, one cannot find a pair of independent sets $I$ and $J$ with size
$\lfloor ns\rfloor$, such that their overlap (intersection) has cardinality at least ${n(1-z)\log d\over d}$ and at most ${n(1+z)\log d\over d}$.

Note that for all $\beta > 1/\sqrt{2}$, there exists $z$
satisfying $0 < z < \sqrt{2\beta^2 - 1}$ and so the theorem is
not vacuous in this setting. Furthermore as $\beta \to 1$, $z$
can be chosen arbitrarily close to $1$ making the forbidden
overlap region extremely broad.
That is, as the size of the independent sets in consideration approaches the maximum possible (namely as $\beta\uparrow 1$),
and as $d\rightarrow\infty$, we
can take $z\rightarrow 1$.
In other words, with probability approaching one, two nearly largest independent sets either overlap
almost entirely or almost do not have an intersection. This is the key result for establishing our hardness bounds for existence of local algorithms.

A slightly different version of the first of these results can be found as Lemma 12 in~\cite{coja2011independent}.
The latter paper shows that if an independent set $I$ is chosen uniformly at random from the set with size nearly $(1+\beta)n\log d/d$,
then with high probability (with respect to the choice of $I$), there exists  an empty overlap region in the sense described above. In fact, this empty
overlap region exists for every $\beta\in (0,1)$, as opposed to just $1>\beta>1/2+1/(2\sqrt{2})$ as in our case. Unfortunately, this result cannot be used for our purposes,
since this result does not rule out the existence of rare sets $I$ for which no empty overlap exists.

\paragraph{Overlapping from local algorithms}

Next we turn to the formalizing the notion of using a local
function $f_r$ twice on coupled randomness to produce overlapping
independent sets.

Fix an $r$-local independence function $f_r$.
Given a vector ${\bf X}=(X_u, 1\le u\le n)$ of variables $X_u\in [0,1]$,
recall that
$I_{\G}(f_r,{\bf X})$ denotes the independent set of $\G$ given
by $u \in
I_{\G}(f_r,{\bf X})$ if and only if $f_r(u,\G,{\bf X}) = 1$.

Recall that ${\bf X}$ is chosen according to the uniform distribution on $[0,1]^n$. Namely,
$X_u$ are independent and uniformly distributed  over $[0,1]$.
In what follows we consider some joint distributions on
pairs of vectors $({\bf X}, {\bf Y})$ such that marginal
distributions on the vector ${\bf X}$ and ${\bf Y}$ are
uniform on $[0,1]^n$, though ${\bf X}$ and ${\bf Y}$ are dependent on each other.
The intuition behind the proof of Theorem~\ref{theorem:MainResult} is as follows.
Note that if ${\bf X} = {\bf Y}$ then $I_{\G}(f_r,{\bf X})=I_{\G}(f_r,{\bf Y})$. As a result
the overlap $I_{\G}(f_r,{\bf X})\cap I_{\G}(f_r,{\bf Y})$ between $I_{\G}(f_r,{\bf X})$ and $I_{\G}(f_r, {\bf Y})$
is $\alpha(f_r)n+o(n)$ in expectation.
On the other hand, if ${\bf X}$ and ${\bf Y}$ are independent,
then the overlap between $I_{\G}(f_r,{\bf X})$ and $I_{\G}(f_r, {\bf Y})$
is $\alpha^2(f_r)n + o(n)$ in expectation,
since the decision to pick a vertex $u$ in $I$ is independent
for most vertices when ${\bf X}$ and ${\bf Y}$ are independent.
(In particular, note that if the local neighborhood around
$u$ is a tree, which according to Proposition~\ref{prop:local-tree} happens with probability approaching unity, then the two
decisions are independent, and $u \in I$ with probability $\alpha(f_r)$.)
Our main theorem shows that by coupling the variables, the
overlap can be arranged to be of any intermediate size, to
within an additive $o(n)$ factor. In particular, if $\alpha(f_r)$ exceeds ${1\over 2}+{1\over 2\sqrt{2}}$ we will be able
to show that the overlap can be arranged to be between the values  ${(1 - z)n\log d\over d}$ and ${(1 + z)n\log d\over d}$, described in Theorem~\ref{thm:cluster}
which contradicts the statement of this theorem.

\begin{theorem}\label{thm:couple}
Fix a positive integer $d$. For constant $r$, let
$f_r(u,\G,\bfx)$ be an $r$-local independence function
and let $\alpha = \alpha(f_r)$.
For every $\gamma \in [\alpha^2,\alpha]$ and $\epsilon > 0$,
and for every sufficiently large $n$, there exists a distribution
on variables $({\bf X}, {\bf Y}) \in [0,1]^n \times [0,1]^n$
such that
$$\pr_{\G_d(n),({\bf X},{\bf Y})}
\left( |I_{\G_d(n)}(f_r,{\bf X}) \cap I_{\G_d(n)}(f_r,{\bf Y})| \not\in
[(\gamma-\epsilon) n, (\gamma+\epsilon)n] \right) \leq \epsilon.$$
\end{theorem}

\section{Proof of Theorem~\ref{theorem:MainResult}}\label{section:ProofOfMainResult}

We now show how Theorems~\ref{thm:cluster} and \ref{thm:couple} immediately
imply
Theorem~\ref{theorem:MainResult}.

\begin{proof}[Proof of Theorem~\ref{theorem:MainResult}]
Fix an $r$-local function $f_r$ and let $\alpha = \alpha(f_r)$.
Fix $0 < \eta < 1$. We will prove below that for sufficiently large
$d$ we have $\alpha/\alpha_d \leq 1/2 + 1/(2\sqrt{2}) + \eta$. The theorem
will then follow.

Let $\epsilon = \frac{\eta \log d}{2d}$.
By Proposition~\ref{prop:variance} we have that almost surely
an independent set returned by $f_r$ on $\G_d(n)$ is of size
at least $(\alpha-\epsilon)n$. Furthermore for every
$\gamma \in [\alpha^2,\alpha]$ we have,
by Theorem~\ref{thm:couple}, that $\G_d(n)$ almost surely has
two independent sets $I$ and $J$, with
\begin{equation}
\label{eq:lower-bound}
|I|, |J| \geq (\alpha-\epsilon) n
\mbox{ and } |I \cap J| \in [(\gamma-\epsilon)n, (\gamma+\epsilon)n].
\end{equation}
Finally, by Theorem~\ref{theorem:LargestIS}, we have that for sufficiently
large $d$, $|I|,|J| \leq (2d^{-1}\log d)(1+ \eta)n\leq 4d^{-1}\log d n$ and so
$\alpha^2 \leq d^{-1}\log d$, allowing us to set
$\gamma = d^{-1}/\log d$.

Now we apply Theorem~\ref{thm:cluster} with $z = \epsilon d/\log d$
and
$\beta > \sqrt{\frac{1+z^2}2}$. (Note that for this choice we
have $z < 1$ and $z < \sqrt{2\beta^2 - 1} < \beta < 1$.
We will also use later the fact that for this choice
we have $\beta \leq 1/\sqrt{2} + z = 1/\sqrt{2} + \epsilon d^{-1}\log d$.)
Theorem~\ref{thm:cluster} asserts that almost surely $\G_d(n)$
has no independent sets of size at least $(1+\beta)d^{-1}\log d n$
with intersection size in $[(1-z)d^{-1}\log d n, (1+z)d^{-1}\log d n]$.
Since
$|I \cap J| \in [(\gamma-\epsilon)n, (\gamma+\epsilon)n] =
[(1-z)d^{-1}\log d n, (1+z)d^{-1}\log d n]$, we conclude that
$\min\{|I|, |J|\} \leq
(1+\beta)d^{-1}\log d n$.
Combining with Equation (\ref{eq:lower-bound}) we get that
$(\alpha - \epsilon)n \leq \min\{|I|,|J|\} \leq (1 + \beta)d^{-1}\log d n$
and so $\alpha \leq (1 + \beta)d^{-1}\log d + \epsilon$,
which by the given bound on $\beta$ yields
$\alpha \leq (1 + 1/\sqrt{2})d^{-1}\log d + 2\epsilon =
(1 + 1/\sqrt{2} + \eta) d^{-1}\log d$.
On the other hand we also have $\alpha_d \geq (2 - \eta)d^{-1}\log d$.
It follows that $\alpha/\alpha_d \leq 1/2 + 1/2\sqrt{2} + \eta$ as desired.

\end{proof}

\section{Proof of Theorem~\ref{thm:couple}}\label{sec:couple}

For parameter $p\in [0,1]$, we define the $p$-correlated distribution
on vectors of random variables $(\bfX,\bfY)$ to be the following:
Let $\bfX, \bfZ$ be independent uniform vectors over $[0,1]^n$.
Now let $Z_u = X_u$ with probability $p$ and $Y_u$ with probability
$1-p$ independently for every $u \in V(G)$.

Let $f(u,\G,\bfx)$ and $\alpha$ be as in the theorem statement.
Recall that $f(\bfx) = f(1,\T_{d,r},\bfx)$ is the decision of $f$
on the canonical tree of degree $d$ and depth $r$ rooted at
the vertex $1$.
Let $\gamma(p)$ be the probability that $f(\bfX) = 1$ and
$f(\bfY) = 1$, for $p$-correlated variables $(\bfX,\bfY)$.
As with Proposition~\ref{prop:variance} we have
the following.

\begin{lemma}
\label{lem:correlated}
For every $d$, $r$, $\epsilon > 0$ and $r$-local function $f$, for
sufficiently large $n$ we have:
$$\pr_{\G_d(n),(\bfX,\bfY)} \left( | | I_{\G_d(n)}(f,\bfX) \cap I_{\G_d(n)}(f,\bfY) | - \gamma(p)
\cdot n | \geq \epsilon n \right) \leq \epsilon,$$
where $(\bfX,\bfY)$ are $p$-correlated distributions on $[0,1]^n$.
\end{lemma}

\begin{proof}
By Proposition~\ref{prop:local-tree} we have that almost surely almost all
local neighborhoods are trees and so for most vertices $u$
the probability that $u$ is chosen to be in the independent sets
$I(f,\bfX)$ and $I(f,\bfY)$ is $\gamma(p)$. By linearity of
expectations we get that $\E[|I(f,\bfX) \cap  I(f,\bfY)|] = \gamma(p)\cdot n
+ o(n)$.
Again observing that most local neighborhoods are disjoint we have
that the variance of $|I(f,\bfX) \cap I(f,\bfY)|$ is $O(n)$.
We conclude, by applying the Chebychev bound, that
$|I(f,\bfX) \cap I(f,\bfY)|$ is concentrated around the expectation
and the lemma follows.
\end{proof}

We also note that for $p = 1$ and $p=0$ the quantity
$\gamma(p)$ follow immediately from their definition.

\begin{proposition}
$\gamma(1) = \alpha$ and $\gamma(0) = \alpha^2$.
\end{proposition}

Now to prove Theorem~\ref{thm:couple} it suffices to prove that
for every $\gamma \in [\alpha^2,\alpha]$ there exists a
$p$ such that $\gamma(p)  = \gamma$. We show this next
by showing that $\gamma(p)$ is continuous.

\begin{lemma}\label{lemma:gamma-p-continuous}
For every $r$, $\gamma(p)$ is a continuous function of $p$.
\end{lemma}
\begin{proof} Let $(W_u, u\in \T_{d,r})$ be random variables associated with nodes in $\T_{d,r}$, uniformly distributed over $[0,1]$,
which are independent for different $u$ and also independent from $X_u$ and $Z_u$.
We use $W_u$ as generators for the events $Y_u=X_u$ vs $Y_u=Z_u$. In particular, given $p$, set $Y_u=X_u$ if $W_u\le p$ and $Y_u=Z_u$ otherwise. This process is exactly
the process of setting variables $Y_u$ to $X_u$ and $Z_u$ with probabilities $p$ and $1-p$ respectively, independently for all nodes $u$.
Now fix any $p_1<p_2$,
and let $\delta<(p_2-p_1)/d^{r+1}$.
We use the notation $f_r(X_u,Z_u,W_u,p)$ to denote the value of $f_r$ when the seed variables realization is $(W_u, u\in\T_{d,r})$, and the threshold value $p$ is used.
Namely, $f_r(X_u,Z_u,W_u,p)=f_r\left(X_u\mb{1}\{W_u\le p\}+Z_u\mb{1}\{W_u>p\}, u\in\T_{d,r}\right)$.
Here, for ease of notation,
the reference to the tree $\T_{d,r}$ is dropped. Utilizing this notation we have
\begin{align*}
\gamma(p)=\pr\left(f_r(X_u)=f_r(X_u,Z_u,W_u,p)=1\right).
\end{align*}
Therefore,
\begin{align*}
\gamma(p_2)-\gamma(p_1)&=\pr\left(f_r(X_u)=f_r(X_u,Z_u,W_u,p_2)=1\right)-\pr\left(f_r(X_u)=f_r(X_u,Z_u,W_u,p_1)=1\right)\\
&=\E[f_r(X_u)f_r(X_u,Z_u,W_u,p_2)-f_r(X_u)f_r(X_u,Z_u,W_u,p_1)].
\end{align*}
Observe that the event $W_u\notin [p_1,p_2]$ for all $u\in\T_{d,r}$ implies $f_r(X_u,Z_u,W_u,p_1)=f_r(X_u,Z_u,W_u,p_2)$ for every
realization of $X_u$ and $Z_u$.
Therefore, by the union bound and since $|\T_{d,r}|<d^{r+1}$, we have
\begin{align*}
|\gamma(p_2)-\gamma(p_1)|\le d^{r+1}(p_2-p_1).
\end{align*}
Since $r$ is fixed, the continuity of $\gamma(p)$ is established.
\end{proof}

We are now ready to prove Theorem~\ref{thm:couple}.

\begin{proof}[Proof of Theorem~\ref{thm:couple}]
Given $\gamma \in [\alpha^2,\alpha]$ by
Lemma~\ref{lemma:gamma-p-continuous} we have that there exists a $p$
such that $\gamma = \gamma(p)$.
For this choice of $p$, let $(\bfX,\bfY)$ be a pair of $p$-correlated
distributions.
Applying Lemma~\ref{lem:correlated} to this choice of $p$, we get that with
probability at least $1-\epsilon$ we have
$| I_{\G_d(n)}(f,\bfX) \cap I_{\G_d(n)}(f,\bfY) | \in [(\gamma- \epsilon) n,
(\gamma+\epsilon)n]$ as desired.
\end{proof}

\section{Theorem~\ref{theorem:ISclustering}: Case of the \ER graph $\G(n,d/n)$}\label{section:Proof.ER}
In this section we prove Theorem~\ref{theorem:ISclustering} for the case of the
random \ER graph. Specifically we show that the overlap of two independent sets of
near maximum cardinality can not be of some intermediate sizes.

The proof is based on a simple moment argument. We first determine
the expected number of pairs of independent sets with a prescribed overlap size
and show that this expectation converges to zero as $n \to \infty$ and in fact
converges to zero exponentially fast when the overlap size falls into
the corresponding inverval. The result then follows from Markov inequality.

Fix positive integers $k\le m\le n$. Recall that $\Overlap(n,d,m,k)$ is the
set of all pairs of independent sets of cardinality $m$ with intersection size
$k$ in the random graph $\G(n,d/n)$.
It is straightforward to see that
\begin{align}\label{eq:factorials}
\E[|\Overlap(n,d,m,k)|]={n! \over  k! (m-k)! (m-k)! (n-2m+k)!}\left(1-{d\over n}\right)^{{2m-k \choose 2}-(m-k)^2}
\end{align}
Let $m=\lfloor ns \rfloor$, where we remind that $s = (1+\beta)d^{-1}\log d$
is given by the statement of the theorem.
Set $k=\lfloor nx \rfloor$ for any
\begin{align}\label{eq:xinterval}
x\in \left({(1-z)\log d\over d},{(1+z)\log d\over d}\right)
\end{align}
It suffices to show that there exists $\gamma>0$ such that
\begin{align}\label{eq:UpperBoundGamma}
\limsup_{n\rightarrow\infty}n^{-1}\log\E[|\Overlap(n,d,\lfloor ns \rfloor,\lfloor nx \rfloor)|]\le -\gamma,
\end{align}
for all $x$ in the interval (\ref{eq:xinterval}),
as then we can use a union bound on the integer choices
\begin{align*}
k\in \left(n{(1-z)\log d\over d},n{(1+z)\log d\over d}\right).
\end{align*}
From this point on we ignore $\lfloor\cdot\rfloor$ notation for the ease of exposition. It should be clear that this does not affect the argument. From~(\ref{eq:factorials}), after simplifying using Stirling's approximation ($a! \approx (a/e)^a$) and the fact that
$\ln (1 - y) \approx - y$ as $y \to 0$, we have
\begin{align}
\limsup_nn^{-1}\log&\E[|\Overlap(n,d,\lfloor ns \rfloor,\lfloor nx \rfloor)|] \notag\\
&=x\log x^{-1}+2(s-x)\log(s-x)^{-1}+(1-2s+x)\log(1-2s+x)^{-1} \notag\\
&-d\left({(2s-x)^2\over 2}-(s-x)^2\right)\label{eq:FinalMax}
\end{align}
We further simplify this expression as
\begin{align*}
x\log x^{-1}+2(s-x)\log(s-x)^{-1}+(1-2s+x)\log(1-2s+x)^{-1}
-ds^2+dx^2/2.
\end{align*}
We have from (\ref{eq:xinterval}) that for large enough $d$
\begin{align*}
x^{-1}&\le d.
\end{align*}
Also, for large enough $d$, since $z<\beta$, then
\begin{align*}
(s-x)^{-1}&\le \left({(1+\beta)\log d\over d}-{(1+z)\log d\over d}\right)^{-1}\le d.
\end{align*}
Finally, we use the following following asymptotics valid as $d\rightarrow\infty$:
\begin{align}
(1-2s+x)\log(1-2s+x)^{-1}=O\left({\log d\over d}\right), \label{eq:OneMinus2sPlusx}
\end{align}
which applies since $0\leq x\le s=O(\log d/d)$. Substituting the expression for $s = (1+\beta)d^{-1}\log d$,  we obtain a bound
\begin{align*}
n^{-1}\log\E[|\Overlap(ns,nx)|]&\le x\log d+2\left({(1+\beta)\log d\over d}-x\right)\log d+O(\log d/d) \\
&-d\left({(1+\beta)\log d\over d}\right)^2+dx^2/2.
\end{align*}
Writing $x=(1+\hat z)\log d/d$,
where according to (\ref{eq:xinterval}) $\hat z$ varies in the interval $[-z,z]$,
we can conveniently rewrite our bound as
\begin{align*}
{\log^2d \over d}\Big( 2(1+\beta)-(1+\beta)^2-(1+\hat z)+(1+\hat z)^2/2\Big)+O(\log d/d).
\end{align*}
Now we can force the expression to be negative for large enough $d$, provided that
\begin{align*}
2(1+\beta)-(1+\beta)^2-(1+\hat z)+(1+\hat z)^2/2<0,
\end{align*}
which is equivalent to $|\hat z|<\sqrt{2\beta^2-1}$ which in turn follows
from the conditions on $z$
in the hypothesis of the theorem statement.

This completes the proof of (\ref{eq:NonOverlap}) and thus the proof of the theorem
for the case of \ER graph.

\section{Theorem~\ref{theorem:ISclustering}: Case of the random regular graph $\G_d(n)$}\label{section:RandomRegGraph}
We now turn to the case of random regular graphs $\G_n(d)$. We use a configuration model of $\G_d(n)$~\cite{BollobasBook},\cite{JansonBook},
which is obtained by replicating each of the $n$
nodes of the graph $d$ times, and then creating a  random uniformly chosen matching connecting these $dn$ nodes. Since $nd$ is assumed to be even, such a matching
exists. Then for every two nodes  $u,v\in [n]$
an edge is created between $u$ and $v$, if there exists at least one edge between any of the replicas of $u$ and any of the replicas of $v$. This step
of creating edges between nodes in $[n]$ from the matching on $nd$ nodes we call projecting.
It is known that, conditioned on the absence of loops and parallel edges, this gives a model of a random regular graph. It is also known
that the probability of appearing of at least one loop or at least two parallel edges is bounded away from zero when $d$ is bounded. Since we are only concerned with statements
taking place with high probability, such a conditioning is irrelevant to us and thus we assume that $\G_d(n)$ is obtained simply by taking a random uniformly chosen matching
and projecting.
The configuration model is denoted by $\bar \G_d(n)$, with nodes denoted by $(i,r)$ where $i=1,2,\ldots,n$ and $r=1,\ldots,d$. Namely,
$(i,r)$ is the $r$-th replica of node $i$ in the original graph. Given any set $A\subset [n]$, let
$\bar A$ be the natural extension of $A$ into the configuration model. Namely $\bar A=\{(i,r): i\in I, r=1,\ldots,d\}$.

Recall that $\Overlap_d(n,m,k)$ stands for the set of pairs of independent sets $I,J$
in $\G_d(n)$ such
that $|I|=|J|=m$ and $|I\cap J|=k$.
Note that there are possibly some edges between $\bar I\setminus \bar J$ and $\bar J\setminus \bar I$
resulting in edges between $I\setminus J$ and $J\setminus I$. Let
$\ROverlap(m,k,l)\subset \Overlap_d(n,m,k)$ be the set of pairs $I,J$ such that the number of edges
between $\bar I\setminus \bar J$ and $\bar J\setminus \bar I$ in the configuration graph model $\bar \G_d(n)$ is exactly $l$.
Here, for the ease of notation we dropped the references to $d$ and $n$. Observe that $l$ is at most $d(m-k)$
and $\cup_{l =0}^{d(m-k)} \ROverlap(m,k,l) = \Overlap_d(n,m,k)$. In what follows we will bound the expected size
of $\ROverlap(m,k,l)$ for every $l$, and thus the expected size of their union.

For $(I,J) \in \ROverlap(m,k,l)$ the number of edges between the set $I\cup J$ and its complement $[n]\setminus (I\cup J)$ is precisely
$(2m-k)d-2l$, since $|I\cup J|=2m-k$. The same applies to the configuration model: the number of edges between $\bar I\cup\bar J$ and its complement
$[nd]\setminus (\bar I\cup\bar J)$ is precisely $(2m-k)d-2l$.
The value of  $\E[|\ROverlap(m,k,l)|]$ is then computed as follows.  Let $R=2m-k$ and $l\le d(m-k)$.
\begin{lemma}\label{lemma:Amkl}
\begin{align*}
\E|\ROverlap(m,k,l)|&={n\choose k, m-k, m-k, n-R}{md-kd \choose l}^2{nd-Rd \choose Rd-2l}l!(Rd-2l)!\times\\
&\times {(nd-2Rd+2l)!\over (nd/2-Rd+l)!2^{nd/2-Rd+l}}{(nd/2)!2^{nd\over 2}\over (nd)!}.
\end{align*}
\end{lemma}

\begin{proof}
The proof is based on the fact that the number of matchings on a set of $m$ nodes (for even $m$) is ${m!\over (m/2)!2^{m\over 2}}$.
So the term ${(nd/2)!2^{nd\over 2}\over (nd)!}$ is precisely the inverse of the number of configuration graphs $\bar \G_d(n)$. The term
${n\choose k, m-k, m-k, n-R}$ is the number of ways of selecting a pair of sets $I$ and $J$ with cardinality $m$ each and intersection size $k$.
Finally,
\begin{align*}
{md-kd \choose l}^2{nd-Rd \choose Rd-2l}l!(Rd-2l)!
{(nd-2Rd+2l)!\over (nd/2-Rd+l)!2^{nd/2-Rd+l}}
\end{align*}
is the number of graphs $\G_d(n)$ such that for a given choice of sets $I$ and $J$, both sets are independent sets, and the number of edges between
$I\setminus J$ and $J\setminus I$ is $l$. Here ${md-kd \choose l}^2$ represents the number of choices for end points of the $l$ edges between $I\setminus J$ and $J\setminus I$;
$l!$ represents the number of matchings once these choices are made;
${nd-Rd \choose Rd-2l}$ represents the number of choices for the end points of edges connecting $I\cup J$ with its complement;
$(Rd-2l)!$ represents the number of matchings once these choices are made; and finally
\begin{align*}
{(nd-2Rd+2l)!\over (nd/2-Rd+l)!2^{nd/2-Rd+l}}
\end{align*}
represents the number of matching choices between the remaining $nd-2Rd+2l$ nodes in the complement of $\bar I\cup \bar J$.
\end{proof}

We write $k=xn, m=sn, l=dyn$, where $x\le s\le 1$. Then $R=(2s-x)n$ and $y\le s-x$. Our main goal is establishing the following analogue of
(\ref{eq:UpperBoundGamma}).

\begin{lemma}\label{lemma:UpperBoundGammaReg}
There exists $\gamma>0$ such that
\begin{align}\label{eq:UpperBoundGammaReg}
\limsup_{n\rightarrow\infty}n^{-1}\log\E[|\ROverlap(\lfloor ns \rfloor,\lfloor nx \rfloor,\lfloor ny \rfloor)|]\le -\gamma,
\end{align}
for $s=(1+\beta)d^{-1}\log d$, for all $x$ in the interval (\ref{eq:xinterval}) and all $0\le y\le s-x$.
\end{lemma}
The claim (\ref{eq:NonOverlapRegular}) of Theorem~\ref{theorem:MainResult}
follows from Lemma~\ref{lemma:UpperBoundGammaReg} by an argument similar to the one for the \ER graph.
The rest of this section is devoted to proving Lemma~\ref{lemma:UpperBoundGammaReg}.

By Lemma~\ref{lemma:Amkl}, we have
\begin{eqnarray*}
\E[|\ROverlap(m,k,l)|]&= &{n\choose k, m-k, m-k, n-R}{md-kd \choose l}^2{nd-Rd \choose Rd-2l}l!(Rd-2l)!(1+o(1))\\
&& \times{(nd-2Rd+2l)^{(nd-2Rd+2l)\over 2}\over e^{(nd-2Rd+2l)\over 2}}
{e^{nd\over 2}\over (nd)^{nd\over 2}}(1+o(1))\\
&=&{n!\over k! ((m-k)!)^2 (n-R)!} {((md-kd)!)^2\over (l!)^2((md-kd-l)!)^2}
{(nd-Rd)!\over (Rd-2l)!(nd-2Rd+2l)!}\\
&&\times
l!(Rd-2l)!(nd-2Rd+2l)^{(nd-2Rd+2l)\over 2}e^{Rd-l}(nd)^{-{nd\over 2}}(1+o(1))\\
&=&{n!\over k! ((m-k)!)^2 (n-R)!} {((md-kd)!)^2\over l!((md-kd-l)!)^2}
{(nd-Rd)!\over (nd-2Rd+2l)!}\\
&& \times
(nd-2Rd+2l)^{(nd-2Rd+2l)\over 2}e^{Rd-l}(nd)^{-{nd\over 2}}(1+o(1))
\end{eqnarray*}

We now consider the logarithm of the expression above normalized by $n$.  Thus
\begin{eqnarray*}
\lefteqn{n^{-1}\log\E[|\ROverlap(m,k,l)|]}\\
&=& -x\log x-2(s-x)\log(s-x)-(1-2s+x)\log(1-2s+x)\\
&&+2(sd-xd)\log(sd-xd)-2(sd-xd)-dy\log dy+dy\\
&&-2(sd-xd-dy)\log(sd-xd-dy)+2(sd-xd-dy)\\
&&+(d-2ds+dx)\log(d-2ds+dx)-(d-2ds+dx)\\
&&-(d-4ds+2dx+2dy)\log(d-4ds+2dx+2dy)+(d-4ds+2dx+2dy)\\
&&+{1\over 2}(d-4ds+2dx+2dy)\log(d-4ds+2dx+2y)\\
&&+d(2s-x-y)-{d\over 2}\log d\\
&=&-x\log x-2(s-x)\log(s-x)-(1-2s+x)\log(1-2s+x)\\
& &+2(sd-xd)\log(sd-xd)-dy\log dy\\
& &-2(sd-xd-dy)\log(sd-xd-dy)\\
& &+(d-2ds+dx)\log(d-2ds+dx)\\
& &-{1\over 2}(d-4ds+2dx+2dy)\log(d-4ds+2dx+2dy)-{d\over 2}\log d\\
& &-2(sd-xd)+dy+2(sd-xd-dy)-(d-2ds+dx)\\
& &+(d-4ds+2dx+2dy)+d(2s-x-y)
\end{eqnarray*}
The term not involving $\log$ is easily checked to be zero. Consider terms of the form $\log(d A)=\log d+\log A$ and consider the $\log d$ terms.
The corresponding multiplier is
\begin{eqnarray*}
&2(sd-xd)-dy-2(sd-xd-dy)+(d-2ds+dx)
-{1\over 2}(d-4ds+2dx+2dy)-{d\over 2},
\end{eqnarray*}
which again is found to be zero. The final expression we obtain is then
\begin{eqnarray}
&= & -x\log x-2(s-x)\log(s-x)-(1-2s+x)\log(1-2s+x)\notag\\
& &+2d(s-x)\log(s-x)-dy\log y\notag\\
& &-2d(s-x-y)\log(s-x-y)\notag\\
& &+d(1-2s+x)\log(1-2s+x)\notag\\
& &-{d\over 2}(1-4s+2x+2y)\log(1-4s+2x+2y). \label{eq:BoundRegular}
\end{eqnarray}
We now recall that $s=(1+\beta)\log d/d$ and  $x$ lies in the interval (\ref{eq:xinterval}). We consider now two cases. Specifically, we first consider the case
\begin{align}\label{eq:yLarge}
(\beta+z+1)^2{\log^2d\over d^2} \le y\le s-x,
\end{align}
and then consider the case
\begin{align}\label{eq:ySmall}
0\le y\le (\beta+z+1)^2{\log^2d\over d^2} .
\end{align}
Assume first that (\ref{eq:yLarge}) holds.
Consider the terms containing $y$:
\begin{align*}
f(y)\triangleq -dy\log y-2d(s-x-y)\log(s-x-y)-{d\over 2}(1-4s+2x+2y)\log(1-4s+2x+2y)
\end{align*}
Then
\begin{align*}
{d^{-1}\dot f(y)}&=-\log y-1+2\log(s-x-y)+2-\log(1-4s+2x+2y)-1\\
&=-\log y+2\log(s-x-y)-\log(1-4s+2x+2y).
\end{align*}
Now by our assumption (\ref{eq:yLarge}), we have $y\ge (\beta+z+1)^2 d^{-2}\log^2d$ implying
\begin{align*}
-\log y\le -2\log(\beta+z+1)+2\log d-2\log\log d
\end{align*}
Also $4s-2x-2y\le 4s<8\log d/d=O(\log d/d)$, implying that
$\log(1-4s+2x+2y)=O(\log d/d)$. Finally, from (\ref{eq:xinterval}) we have
$s-x-y\le s-x=(\beta+z)\log d/d$, implying that $\log(s-x-y)\le -\log d+\log\log d+\log(\beta+z)$.
Combining, we obtain that
\begin{align*}
{d^{-1}\dot f(y)}&\le-2\log(\beta+z+1)+2\log d-2\log\log d-2\log d+2\log\log d+2\log(\beta+z)\\
&+O(\log d/d)\\
&=-2\log(\beta+z+1)+2\log(\beta+z)+O(\log d/d).
\end{align*}
In particular, the derivative is negative for large enough $d$ and thus the largest value of $f(y)$ when $y$ is in the interval (\ref{eq:yLarge})
is obtained at the left end of this interval. Thus, without the loss of generality, we may assume from now on that the bound (\ref{eq:ySmall}) holds.

For convenience we start with the term $(1-2s+x)\log(1-2s+x)$ in (\ref{eq:BoundRegular}).
Using the first order Taylor approximation  $\log(1-t)=-t+o(t)$,  and the fact $s=O(\log d/d), x=O(\log d/d)$,  we have
\begin{align*}
(1-2s+x)\log(1-2s+x)
&=O(\log d/d)\\
&=o(\log^2 d/d).
\end{align*}
Next we analyze the term $d(1-2s+x)\log(1-2s+x)$. Using the approximation
\begin{align*}
(1-t)\log(1-t)=-t+t^2/2+O(t^3),
\end{align*}
we obtain
\begin{align*}
d(1-2s+x)&\log(1-2s+x)
=-d(2s-x)+{d\over 2}(2s-x)^2+O(d(2s-x)^3).
\end{align*}
Before we expand this term in terms of $d$, it will be convenient to obtain a similar expansion for the last term in (\ref{eq:BoundRegular})
\begin{align*}
{d\over 2}&(1-4s+2x+2y)\log(1-4s+2x+2y)\\
&=-{d\over 2}(4s-2x-2y)+{d\over 4}(4s-2x-2y)^2+O(d(4s+2x+2y)^3)\\
&=-d(2s-x)+dy+d(2s-x)^2-2d(2s-x)y+dy^2+O(d(4s+2x+2y)^3)\\
\end{align*}
Applying the upper bound (\ref{eq:ySmall}) we have  $O(d(2s-x)^3)=O(\log^3d/d^2)=o(\log^2 d/d)$, $O(d(4s+2x+2y)^3)=o(\log^2 d/d)$, and $dy^2=O(\log^4 d/d^3)=o(\log^2 d/d)$.
Combining,  we obtain
\begin{align}
d(1-2s+x)&\log(1-2s+x)-{d\over 2}(1-4s+2x+2y)\log(1-4s+2x+2y)\notag\\
&=-{d\over 2}(2s-x)^2-dy+2d(2s-x)y+o(\log^2 d/d)\notag\\
&=-{d\over 2}(2s-x)^2-dy+o(\log^2 d/d), \label{eq:dy}
\end{align}
where again applying bound (\ref{eq:ySmall}) on $y$ we have used
\begin{align*}
2d(2s-x)y=O\left(d{\log d\over d}{\log^2 d\over d^2}\right)=o(\log^2 d/d).
\end{align*}

Next it is convenient to analyze the following two terms together:
\begin{eqnarray*}
\lefteqn{2d(s-x)\log(s-x)-2d(s-x-y)\log(s-x-y)}\\
&= &2d(s-x)\log(s-x)-2d(s-x)\log(s-x-y)+2dy\log(s-x-y)\\
&=& 2d(s-x)\log(s-x)-2d(s-x)\log(s-x)-2d(s-x)\log(1-y(s-x)^{-1})\\
& & + 2dy\log(s-x)-2dy\log(1-y(s-x)^{-1}))\\
&=& -2d(s-x)\log(1-y(s-x)^{-1})+2dy\log(s-x)-2dy\log(1-y(s-x)^{-1}))\\
&=& 2d(s-x)y(s-x)^{-1}+O(dy^2(s-x)^{-1})\\
& &+2dy\log(s-x)+2dy^2(s-x)^{-1}+O(dy^3(s-x)^{-2})\\
&=& 2dy+2dy\log(s-x)+2dy^2(s-x)^{-1}+O(dy^2(s-x)^{-1})+O(dy^3(s-x)^{-2})\\
&=& 2dy+2dy\log(s-x)+o(\log^2 d/d),
\end{eqnarray*}
where we have used the asymptotics $y=O(\log^2 d/d^2)$ implied by (\ref{eq:ySmall}) in the last equality.

We now analyze the remaining terms involving $y$. From (\ref{eq:dy}) we have the term $-dy$. Combining with the asymptotics above
and the remaining term $-dy\log y$ from (\ref{eq:BoundRegular}) we obtain
\begin{align}\label{eq:Maxy}
2dy+2dy\log(s-x)-dy-dy\log y=dy+2dy\log(s-x)-dy\log y.
\end{align}
We compute the maximum value of this quantity in the relevant range of $y$ given by (\ref{eq:ySmall}). The first derivative  of this expression is
\begin{align*}
d+2d\log(s-x)-d-d\log y=2d\log(s-x)-d\log y
\end{align*}
which is positive (infinite) at $y=0$. At $y=(\beta+z+1)^2\log^2 d/d^2$,
the first derivative is
\begin{align*}
2d\log(s-x)&-2d\log(\beta+z+1)-2d\log\log d+2d\log d\\
&\le 2d\log(\beta+z)+2d\log\log d-2d\log d-2d\log(\beta+z+1)-2d\log\log d+2d\log d\\
&=2d\log(\beta+z)-2d\log(\beta+z+1)\\
&<0,
\end{align*}
where the inequality relies on $x\ge (1-z)\log d/d$ implied by (\ref{eq:xinterval}), which gives
\begin{align*}
s-x\le (\beta+z)\log d/d.
\end{align*}
The second derivative is $-d/y$ which is negative since $y\ge 0$. Thus, the function is strictly concave with positive and negative derivatives
at the ends of the relevant interval (\ref{eq:ySmall}). The maximum is then achieved at the unique point $y^*$ where the derivative is zero, namely when $2d\log(s-x)-d\log y^*=0$, giving
\begin{align*}
y^*=(s-x)^2.
\end{align*}
Plugging this into the right-hand side of (\ref{eq:Maxy}) we obtain
\begin{align*}
dy^*+&2dy^*\log(s-x)-dy^*\log y^*\\
&=d(s-x)^2+2d(s-x)^2\log(s-x)-d(s-x)^2\log(s-x)^2\\
&=d(s-x)^2.
\end{align*}
Summarizing, and using (\ref{eq:dy}), we find that the expression in (\ref{eq:BoundRegular}) is at most
\begin{align*}
&=-x\log x-2(s-x)\log(s-x)-{d\over 2}(2s-x)^2+d(s-x)^2+o(\log^2 d/d),
\end{align*}
which is precisely the expression (\ref{eq:FinalMax}) we have derived for the case of \ER graph $\G(n,c/n)$,
with the exception of the term $(1-2s+x)\log(1-2s+x)$, which is $o(\log^2 d/d)$ by (\ref{eq:OneMinus2sPlusx}).
We have obtained the expression we have analyzed for the case of graphs $\G(n,c/n)$, for which we have shown that the expression is
negative for the specified choices of $s$ and $x$ for sufficiently large $d$. This completes the proof of Lemma~\ref{lemma:UpperBoundGammaReg} and of
Theorem~\ref{theorem:ISclustering}.

\section*{Acknowledgements}
The first author wishes to thank Microsoft Research New England for
providing exciting and hospitable environment during his visit in 2012 when this work was conducted.
The same author also wishes acknowledge the general support of NSF grant CMMI-1031332.

\newcommand{\etalchar}[1]{$^{#1}$}
\providecommand{\bysame}{\leavevmode\hbox to3em{\hrulefill}\thinspace}
\providecommand{\MR}{\relax\ifhmode\unskip\space\fi MR }
\providecommand{\MRhref}[2]{%
  \href{http://www.ams.org/mathscinet-getitem?mr=#1}{#2}
}
\providecommand{\href}[2]{#2}


\begin{thebibliography}{ACORT11}

\bibitem[ACORT11]{AchlioptasCojaOghlanRicciTersenghi}
D.~Achlioptas, A.~Coja-Oghlan, and F.~Ricci-Tersenghi, \emph{On the solution
  space geometry of random formulas}, Random Structures and Algorithms
  \textbf{38} (2011), 251--268.

\bibitem[Ald]{Aldous:FavoriteProblemsNew}
D.~Aldous, \emph{Some open problems.
  \emph{http://www.stat.berkeley.edu/$\sim$aldous/ Research/OP/index.html}}.

\bibitem[AS92]{AlonSpencer}
N.~Alon and J.~Spencer, \emph{Probabilistic method}, Wiley, 1992.

\bibitem[BCG13]{borgs2013convergent}
Christian Borgs, Jennifer Chayes, and David Gamarnik, \emph{Convergent
  sequences of sparse graphs: A large deviations approach}, arXiv preprint
  arXiv:1302.4615 (2013).

\bibitem[BCL{\etalchar{+}}12]{BorgsChayesEtAlGraphLimitsII}
C.~Borgs, J.T. Chayes, L.~Lov\'{a}sz, V.T. S\'{o}s, and K.~Vesztergombi,
  \emph{Convergent graph sequences {II}: Multiway cuts and statistical
  physics}, Ann. of Math. \textbf{176} (2012), 151--219.

\bibitem[BCL{\etalchar{+}}08]{BorgsChayesEtAlGraphLimitsI}
\bysame, \emph{Convergent graph sequences {I}: Subgraph frequencies, metric
  properties, and testing}, Advances in Math. \textbf{219} (208), 1801--1851.

\bibitem[BCLK]{BorgsChayesKahnLovasz}
C.~Borgs, J.T. Chayes, L.~Lov\'{a}sz, and J.~Kahn, \emph{Left and right
  convergence of graphs with bounded degree}, http://arxiv.org/abs/1002.0115.

\bibitem[BGT10]{BayatiGamarnikTetali}
M.~Bayati, D.~Gamarnik, and P.~Tetali, \emph{Combinatorial approach to the
  interpolation method and scaling limits in sparse random graphs}, Annals of
  Probability, to appear. Conference version in Proc. 42nd Ann. Symposium on
  the Theory of Computing (STOC) (2010).

\bibitem[Bol85]{BollobasBook}
B.~Bollobas, \emph{Random graphs}, Academic Press, Inc., 1985.

\bibitem[CL12]{csoka2012invariant}
E.~Csoka and G.~Lippner, \emph{Invariant random matchings in cayley graphs},
  arXiv preprint arXiv:1211.2374 (2012).

\bibitem[CO11]{coja2011belief}
A.~Coja-Oghlan, \emph{On belief propagation guided decimation for random
  k-sat}, Proceedings of the Twenty-Second Annual ACM-SIAM Symposium on
  Discrete Algorithms, SIAM, 2011, pp.~957--966.

\bibitem[COE11]{coja2011independent}
A.~Coja-Oghlan and C.~Efthymiou, \emph{On independent sets in random graphs},
  Proceedings of the Twenty-Second Annual ACM-SIAM Symposium on Discrete
  Algorithms, SIAM, 2011, pp.~136--144.

\bibitem[EL10]{elek2010borel}
G.~Elek and G.~Lippner, \emph{Borel oracles. an analytical approach to
  constant-time algorithms}, Proc. Amer. Math. Soc, vol. 138, 2010,
  pp.~2939--2947.

\bibitem[F{\L}92]{frieze1992independence}
A.M. Frieze and T.~{\L}uczak, \emph{On the independence and chromatic numbers
  of random regular graphs}, Journal of Combinatorial Theory, Series B
  \textbf{54} (1992), no.~1, 123--132.

\bibitem[Fri90]{FriezeIndependentSet}
A.~Frieze, \emph{On the independence number of random graphs}, Discrete
  Mathematics \textbf{81} (1990), 171--175.

\bibitem[HKNO09]{hassidimetal}
Avinatan Hassidim, Jonathan~A. Kelner, Huy~N. Nguyen, and Krzysztof Onak,
  \emph{Local graph partitions for approximation and testing}, FOCS, IEEE
  Computer Society, 2009, pp.~22--31.

\bibitem[HLS]{HatamiLovaszSzegedy}
H.~Hatami, L.~Lov{\'a}sz, and B.~Szegedy, \emph{Limits of local-global
  convergent graph sequences}, Preprint at http://arxiv.org/abs/1205.4356.

\bibitem[J{\L}R00]{JansonBook}
S.~Janson, T.~{\L}uczak, and A.~Rucinski, \emph{Random graphs}, John Wiley and
  Sons, Inc., 2000.

\bibitem[KS81]{KarpSipser}
R.~Karp and M.~Sipser, \emph{Maximum matchings in sparse random graphs}, 22nd
  Annual Symposium on Foundations of Computer Science, 1981, pp.~364--375.

\bibitem[Lin92]{Linial}
Nathan Linial, \emph{Locality in distributed graph algorithms}, SIAM J. Comput.
  \textbf{21} (1992), no.~1, 193--201.

\bibitem[LN11]{lyons2011perfect}
R.~Lyons and F.~Nazarov, \emph{Perfect matchings as iid factors on non-amenable
  groups}, European Journal of Combinatorics \textbf{32} (2011), no.~7,
  1115--1125.

\bibitem[LS06]{LovaszSzegedy}
L.~Lov\'{a}sz and B.~Szegedy, \emph{Limits of dense graph sequences}, Journal
  of Combinatorial Theory, Series B \textbf{96} (2006), 933--957.

\bibitem[MM09]{MezardMontanariBook}
M.~Mezard and A.~Montanari, \emph{Information, physics and computation}, Oxford
  graduate texts, 2009.

\bibitem[MMZ05]{mezard2005clustering}
M.~M{\'e}zard, T.~Mora, and R.~Zecchina, \emph{Clustering of solutions in the
  random satisfiability problem}, Physical Review Letters \textbf{94} (2005),
  no.~19, 197205.

\bibitem[NO08]{NguyenOnak}
Huy~N. Nguyen and Krzysztof Onak, \emph{Constant-time approximation algorithms
  via local improvements}, FOCS, IEEE Computer Society, 2008, pp.~327--336.

\bibitem[PR07]{ParnasRon}
Michal Parnas and Dana Ron, \emph{Approximating the minimum vertex cover in
  sublinear time and a connection to distributed algorithms}, Theor. Comput.
  Sci. \textbf{381} (2007), no.~1-3, 183--196.

\bibitem[RTVX11]{ronittetal}
Ronitt Rubinfeld, Gil Tamir, Shai Vardi, and Ning Xie, \emph{Fast local
  computation algorithms}, ICS (Bernard Chazelle, ed.), Tsinghua University
  Press, 2011, pp.~223--238.

\bibitem[Tal10]{TalagrandBook}
M.~Talagrand, \emph{Mean field models for spin glasses: Volume {I}: Basic
  examples}, Springer, 2010.

\bibitem[WJ05]{JordanWainwright}
M.~J. Wainwright and M.~I. Jordan, \emph{A variational principle for graphical
  models}, New Directions in Statistical Signal Processing: From Systems to
  Brain. Cambridge, MA: MIT Press, 2005.

\end{thebibliography}

\end{document}